\documentclass[11pt,leqno]{article}
\usepackage{amsmath,amssymb,amsthm,latexsym}
\usepackage{indentfirst}
\usepackage{amsmath}

\newtheorem{theorem}{\bf \large Theorem}[section]
\newtheorem{PROPOSITION}{\bf \large Proposition}[section]

\newtheorem{remark}{\bf \large Remark}[section]

\title{Rigidity of closed minimal hypersurface  in $\mathbb{S}^5$}
\author {{ Pengpeng Cheng,~~~Tongzhu Li,} \\
\small{Department of Mathematics, Beijing Institute of
Technology, Beijing, 100081, China.} \\
\small{ E-mail: ~3120235951@bit.edu.cn,~~litz@bit.edu.cn.}}
\date{}

\begin{document}
\maketitle
\begin{abstract}
Let $M^4\to \mathbb{S}^5$ be a closed immersed minimal hypersurface with constant squared length of
the second fundamental form $S$ in a $5$-dimensional sphere $\mathbb{S}^5$.
In this paper, we prove  that if  $3$-mean curvature $H_3$ and the number $g$ of the distinct principal curvatures are constant, then $M^4$ is an isoparametric hypersurface, and the value of  $S$ can only be $0, 4, 12$. This result supports Chern Conjecture.
\end{abstract}

\medskip\noindent
{\bf 2020 Mathematics Subject Classification:} 53C40, 53C42.
\par\noindent {\bf Key words:}  minimal hypersurface, Chern Conjecture, isoparametric hypersurface, the second fundamental form.

\section{Introduction}
The study of the rigidity of minimal hypersurfaces   is a very interesting problem.  In \cite{sim}, Simons  proved the following well-known Simons integral inequality,
\begin{theorem}(\cite{sim})
Let $M^n$ be a closed minimal submanifold in the unit
sphere $\mathbb{S}^{n+p}$
and $S$ the squared norm of its second fundamental form. Then
$$\int_{M^n}S\Big(S-\frac{n}{2-\frac{1}{p}}\Big)dM\geq 0.$$
In particular, for $S\leq \frac{n}{2-\frac{1}{p}}$,
one has either $S=0$ or $S=\frac{n}{2-\frac{1}{p}}$
identically on M.
\end{theorem}
Shortly afterwards Chern, do Carmo and
Kobayashi (\cite{chern}) and Lawson (\cite{L}) independently showed that Simons' result is sharp and the equality
is realized by the Clifford minimal hypersurfaces in the unit sphere.
\begin{theorem}(\cite{chern},\cite{L})
Let $M^n$ be a closed minimal hypersurface in $(n+1)$-dimensional sphere $\mathbb{S}^{n+1}$ $(n\geq 2)$, and $S$ the squared norm of the second fundamental
form. Then we have
$$\int_{M}S(n-S)dM\geq0.$$
In particular, if  $0\leq S\leq n$, then $S=0$ and $M$ is totally geodesic, or $S=n$ and $M$ is one of  Clifford minimal hypersurfaces.
\end{theorem}

By Gauss equation we have $R=n(n-1)-S$. Based on these, S. S. Chern had insight into the rigidity  and
has proposed the following Chern Conjecture.

{\bf Chern Conjecture}:
{\it Let $x:M^n\to \mathbb{S}^{n+1}$ be a closed  minimal hypersurfaces with constant scalar curvatures $R$ in $(n+1)$-dimensional sphere $\mathbb{S}^{n+1}$($n\geq 2$). Let $\mathbb{A}_R$ be the collection of all the possible values of such scalar curvature $R$, then $\mathbb{A}_R$ is a discrete subset of real numbers.}

The Chern Conjecture remains open, but there are many partial results. Peng and Terng (\cite{peng},\cite{peng1}) made the first effort to solve the Chern Conjecture and confirmed the second gap of $\mathbb{A}_R$. Precisely, they proved that if the scalar curvature $R$ of $M^n$
is a constant, then there exists a positive constant $C(n)$ depending only on $n$ such that if
$n\leq S\leq n+C(n)$, then $S=n$. Later, the pinching constant $C(n)$ was improved to $\frac{n}{3},~n>3$
by Cheng and Yang (\cite{yang2},\cite{yang3}), and to $\frac{3n}{7}$ by Suh and Yang (\cite{yang}), respectively.
In 1993, Chang (\cite{chang2}) solved  Chern Conjecture for $n=3$.
In \cite{de} de Almeida-Brito-Scherfner-Weiss proved that if $M^n (n\geq 4)$ is a closed
minimally immersed hypersurface in $\mathbb{S}^{n+1}$ with constant Gauss-Kronecker curvature and it has three
distinct principal curvatures everywhere, then $M^n$ is an isoparametric hypersurface.  For  $n=4$, Tang-Yang in \cite{tali} proved that, if the scalar curvature $R\geq 0$,  $3$-mean curvature
$H_3$ and the number $g$ of distinct principal curvatures  are constant, then $M^4$ is isoparametric. Tang-Wei-Yan in \cite{tang} and Tang-Yan in \cite{tang1} generalized the theorem of de Almeida
and Brito (\cite{ad}) for $n=3$ to any dimension $n$.
\begin{theorem}(\cite{tang,tang1})\label{tany}
Let $M^n (n\geq 4)$ be a closed immersed hypersurface in $\mathbb{S}^{n+1}$. If the
following conditions are satisfied:\\
(i) $f_k=\sum_{i=1}^n\lambda_i^k, (k=1,\cdots,n-1)$ are constants for the principal curvatures $\lambda_1,\lambda_2, \cdots, \lambda_n$;\\
(ii) the scalar curvature $R\geq 0$;\\
then $M^n$ is isoparametric. Moreover, if $M^n$ has $n$ distinct principal curvatures somewhere, then $R=0$.
\end{theorem}

For $n=4$, Deng-Gu-Wei
in \cite{dgw} dropped the non-negativity assumption of the scalar curvature
under the condition $H_3=0$, and proved  the following theorem,
\begin{theorem}\cite{dgw}\label{wei}
 Any closed minimal Willmore hypersurface $M^n$ of $\mathbb{S}^5$ with constant scalar
curvature must be isoparametric. To be precise, $M^4$ is either an equatorial $4$ sphere, a
product of sphere $\mathbb{S}^2(\frac{\sqrt{2}}{2})\times \mathbb{S}^2(\frac{\sqrt{2}}{2})$
or a Cartan's minimal hypersurface.\\
In particular, $S$ can only be $0, 4, 12$.
\end{theorem}

If $S$ is not constant,  the length of the second fundamental form $S$ has many gaps and that
as analytic geometric object the minimal hypersurfaces have some rigidities. Peng-Terng (\cite{peng},\cite{peng1}) obtained that there
exists a positive constant $\delta(n)$ depending only on the dimension $n$, such that if $n\leq S\leq n+\delta(n), n\leq 5$, then $S\equiv n$.
Later, Cheng-Ishikawa (\cite{cheng}) improved the previous pinching constant when $n\leq 5$, Wei-Xu (\cite{weixu}) extended
the result to $n=6,7$, and Zhang (\cite{zhang}) promoted it to $n\leq 8$. Finally, Ding-Xin (\cite{ding}) proved all the dimensions.
Over the years, there have been many important developments on Chern Conjecture, see for example  \cite{ge}, \cite{li3}, \cite{LS}, \cite{lif} and \cite{scher}.

In this paper, we study the minimal hypersurfaces in $\mathbb{S}^5$ and prove the following main theorem.
\begin{theorem}\label{th1}
Let $x: M^4\to \mathbb{S}^5$ be a closed minimal  hypersurface  with constant squared length of
the second fundamental form $S$  in the $5$-dimensional sphere.
If the $3$-mean curvature $H_3$  and the number $g$ of distinct principal curvatures are constant,
then $M^4$ is isoparametric and $S$ can only be $0, 4, 12$.\\
Furthermore, $M^4$ is
 either an $4$-dimensional equatorial sphere, a product $\mathbb{S}^2(\frac{\sqrt{2}}{2}) \times \mathbb{S}^2(\frac{\sqrt{2}}{2})$, or a Cartan's minimal hypersurface.
\end{theorem}

\begin{remark}
By adding a weaker condition that the number $g$ of distinct principal curvatures is constant,  our main results  remove the condition of the scalar curvature $R\geq 0$ in Theorem \ref{tany} (\cite{tang}), and relax the condition $H_3=0$ in Theorem \ref{wei} (\cite{dgw}) to the condition $H_3=constant$.
\end{remark}

\par\noindent
\section{Preliminaries}

Let $x:M^n\to \mathbb{S}^{n+1}$ be an $n$-dimensional immersed  hypersurface
in an $(n+1)$-dimensional sphere $\mathbb{S}^{n+1}$. For any $p\in M^n$ we choose a local orthonormal frame $\{e_1,\cdots,e_n,e_{n+1}\}$
 around $p$ such that $e_1,\cdots,e_n$ are tangential to $M^n$ and $e_{n+1}$ is normal to $M^n$. Let $\{\omega_1,\cdots,\omega_n\}$ be the dual coframe and $\{\omega_{ij}|~1\leq i,j\leq n\}$ be the connection $1$-forms. In this section we make the following convention
on the range of indices, $$1\leq i,j,k\leq n.$$
Then the structure equations of $M^n$ are given by
\begin{equation}\label{stru}
\begin{split}
&dw_{i}=\sum_jw_{ij}\wedge w_{j}, ~~ w_{ij}+w_{ji}=0,\\
&dw_{ij}=\sum_k w_{ik}\wedge w_{kj}-\frac{1}{2}\sum_{k,l}R_{ijkl}w_{k}\wedge w_{l},
\end{split}
\end{equation}
where $R_{ijkl}$ is the curvature tensor of the induced metric on $M^n$.

Let $II=\sum_{ij}h_{ij}\omega_i\otimes\omega_j$ denote the second fundamental form, $ H=\frac{1}{n}\sum_{i}h_{ii}$ the mean curvature.
The
Gauss equation is
\begin{equation}\label{gauss}
R_{ijkl}=\delta_{ik}\delta_{jl}-\delta_{il}\delta_{jk}+h_{ik}h_{jl}-h_{il}h_{jk},
\end{equation}
and the Codazzi equation is
\begin{equation}\label{coda}
h_{ij,k}=h_{ik,j},
\end{equation}
where the covariant derivative of the second fundamental form is defined by
$$\sum_mh_{ij,m}w_{m}=dh_{ij}+\sum_mh_{mj}w_{mi}+\sum_mh_{im}w_{mj}.$$

The second covariant derivative of the second fundamental form is defined by
$$\sum_mh_{ij,km}w_{m}=dh_{ij,k}+\sum_mh_{mj,k}w_{mi}+\sum_mh_{im,k}w_{mj}+\sum_mh_{ij,m}\omega_{mk}.$$
Thus we have the following Ricci identity
\begin{equation}\label{ricd}
h_{ij,kl}-h_{ij,lk}=\sum_mh_{mj}R_{mikl}+\sum_mh_{im}R_{mjkl}.
\end{equation}

By the Gauss equation (\ref{gauss}), we obtain the Ricci curvature $R_{ij}$ and the scalar curvature $R$ of the hypersurface,
\begin{equation}\label{ricc}
\begin{split}
&R_{ij}=(n-1)\delta_{ij}+nHh_{ij}-\sum_{m}h_{im}h_{mj},\\
&R=n(n-1)+n^2H^2-S,
\end{split}
\end{equation}
where $S=\sum_{i,j}h_{ij}^2$ is the square norm of the second fundamental form.

The shape operator $A$ is the dual  tensor  of the second fundamental form $II$. Let $\{\lambda_1,\cdots,\lambda_n\}$ be the eigenvalues of the shape operator $A$, which are called the principal curvatures of the hypersurface. The principal curvatures $\{\lambda_1,\cdots,\lambda_n\}$ are continuous on $M^n$.  We denote by $g$  the number of the distinct principal curvatures, then $g$ is a local constant. If $g$ is constant on $M^n$, then the principal curvatures $\{\lambda_1,\cdots,\lambda_n\}$ are smooth. If the principal curvatures $\{\lambda_1,\cdots,\lambda_n\}$ are constant,  the hypersurface $x$ is called an isoparametric hypersurface.

Let $\sigma_r : \mathbb{R}^n \to \mathbb{R}$ be the elementary symmetric functions defined  by
$$\sigma_r(x_1, \cdots, x_n) = \sum_{i_1 < i_2 < \cdots < i_r} x_{i_1} x_{i_2} \cdots x_{i_r}.$$
Then $r$-mean curvature $H_r$ of the hypersurface is  defined by
$$H_r = \frac{1}{\binom{n}{r}}\sigma_r(\lambda_1,\cdots,\lambda_n)=\sum_{i_1 < i_2 < \cdots < i_r} \lambda_{i_1} \lambda_{i_2} \cdots \lambda_{i_r}.$$
Let
$$f_k = \text{Tr}(A^k),$$
then $f_1=nH=nH_1,~f_2=\text{Tr}(A^2)=S$.

When $n=4$ and $H=0$, we have the equations,
\begin{equation}\label{niu}
\begin{split}
&f_1 = \sigma_1 = n H_1=0,\\
&f_2 = \sigma_1^2 - 2 \sigma_2 = S, \\
&f_3 = \sigma_1^3 - 3 \sigma_1 \sigma_2 + 3 \sigma_3=3 \sigma_3, \\
&f_4 = \sigma_1^4 - 4 \sigma_1^2 \sigma_2 + 4 \sigma_1 \sigma_3 + 2 \sigma_2^2 - 4 \sigma_4 = \frac{S^2}{2} - 4 \sigma_4.
\end{split}
\end{equation}

\par\noindent
\section{Proof of Theorem \ref{th1}}
Let $x: M^4\to \mathbb{S}^5$ be a closed minimal  hypersurface  with constant squared length of
the second fundamental form $S$  in the $5$-dimensional sphere,
and the $3$-mean curvature $H_3$  and the number $g$ of distinct principal curvatures are constant.
For a point $p\in M^n$, there is   an open  subset $U$ of $M^4$, such that we can choose an orthonormal basis $\{e_1,e_2,e_3,e_4\}$  for $TU$ such that
$$(h_{ij})=diag(\lambda_1,\lambda_2,\lambda_3,\lambda_4).$$
The smooth orthonormal frame $\{e_1,e_2,e_3,e_4\}$ are called  unit principal vectors.

By (\ref{niu}), we have
\begin{equation}\label{cond1}
\begin{split}
&f_1=nH=\lambda_1+\lambda_2+\lambda_3+\lambda_3=0, \\
&f_2=S=\lambda_1^2+\lambda_2^2+\lambda_3^2+\lambda_4^2=constant,\\
&f_3=\lambda_1^3+\lambda_2^3+\lambda_3^3+\lambda_4^3=constant.
\end{split}
\end{equation}

By (\ref{cond1}) we have the following results,
\begin{PROPOSITION}\label{pro1}
Let $x: M^4\to \mathbb{S}^5$ be a closed minimal  hypersurface  with constant squared length of
the second fundamental form $S$  in the $5$-dimensional sphere, and let both
the $3$-mean curvature $H_3$  and the number $g$ of distinct principal curvatures be constant. If  $g\leq 3$, then $M^4$ is an isoparametric hypersurface.
\end{PROPOSITION}

Next we consider $g=4$. Since the principal curvatures $\{\lambda_1,\lambda_2,\lambda_3,\lambda_4\}$ are simple,
 the unit principal vectors $\{e_1,e_2,e_3,e_4\}$ are determined up to a sign.
Let $\{\omega_1,\omega_2,\omega_3,\omega_4\}$ be the dual basis of $\{e_1,e_2,e_3,e_4\}$, and $\{\omega_{ij}|~1\leq i,j\leq 4\}$
the  connection $1$-forms.

By the definition of the covariant derivative of the second fundamental form,
we obtain the following equations,
\begin{equation}\label{con2}
\begin{split}
&e_i(\lambda_j)=e_i(h_{jj})=h_{jj,i},\\
&\omega_{ij}=\sum_m\frac{h_{ij,m}w_{m}}{\lambda_i-\lambda_j},~~i\neq j.
\end{split}
\end{equation}
Combining (\ref{cond1}) and (\ref{con2}), for  any fixed index $i$,
\begin{equation*}
\sum_jh_{jj,i}=0,~~\sum_j\lambda_jh_{jj,i}=0,~~\sum_j\lambda_j^2h_{jj,i}=0.
\end{equation*}
Thus we have
\begin{equation}\label{con3}
\begin{split}
&h_{11,i}= \frac{(\lambda_3 - \lambda_4)(\lambda_2 - \lambda_4)}{(\lambda_1 - \lambda_3)(\lambda_2 - \lambda_1)}h_{44,i}\\
&h_{22,i} = \frac{(\lambda_3 - \lambda_4)(\lambda_1 - \lambda_4)}{(\lambda_1 - \lambda_2)(\lambda_2 - \lambda_3)}h_{44,i}\\
&h_{33,i} = \frac{(\lambda_2 - \lambda_4)(\lambda_1 - \lambda_4)}{(\lambda_1 - \lambda_3)(\lambda_3 - \lambda_2)}h_{44,i}.
\end{split}
\end{equation}

We  introduce the $3$-forms $\{\theta_{ij}|~1\leq i,j\leq 4\}$ defined by
\begin{equation}
\begin{split}
\theta_{12}=\omega_3\wedge\omega_4\wedge\omega_{12},~~\theta_{13}=\omega_4\wedge\omega_2\wedge\omega_{13},
~~\theta_{14}=\omega_2\wedge\omega_3\wedge\omega_{14},\\
\theta_{23}=\omega_1\wedge\omega_4\wedge\omega_{23},~~\theta_{24}=\omega_3\wedge\omega_1\wedge\omega_{24},
~~\theta_{34}=\omega_1\wedge\omega_2\wedge\omega_{34}.
\end{split}
\end{equation}
Since the $\{\omega_1,\omega_2,\omega_3,\omega_4\}$ are determined up to a sign, the $3$-forms $\{\theta_{ij}|~1\leq i,j\leq 4\}$
can be globally well-defined on $M^4$.
Let $d\theta_{ij}=X_{ij}\omega_1\wedge\omega_2\wedge\omega_3\wedge\omega_4$. Combining (\ref{stru}), (\ref{con2}) and (\ref{con3}), we can get the following equations,
\begin{equation}\label{equ1}
\begin{split}
X_{12} &=
		\frac{(\lambda_3 - \lambda_4)\left[(\lambda_1 - \lambda_3)^2(\lambda_2 - \lambda_3) - (\lambda_1 - \lambda_4)^2(\lambda_2 - \lambda_4)\right]h_{44,1}^2}{(\lambda_1 - \lambda_2)^2(\lambda_1 - \lambda_3)^2(\lambda_2 - \lambda_3)^2} \\
		& + \frac{(\lambda_3 - \lambda_4)\left[(\lambda_1 - \lambda_3)(\lambda_2 - \lambda_3)^2 - (\lambda_1 - \lambda_4)(\lambda_2 - \lambda_4)^2\right]h_{44,2}^2}{(\lambda_1 - \lambda_2)^2(\lambda_1 - \lambda_3)^2(\lambda_2 - \lambda_3)^2} \\
		& + \frac{(\lambda_1 - \lambda_4)(\lambda_2 - \lambda_4)(\lambda_3 - \lambda_4)^2h_{44,3}^2}{(\lambda_1 - \lambda_2)^2(\lambda_1 - \lambda_3)^2(\lambda_2 - \lambda_3)^2} + \frac{(\lambda_3 - \lambda_4)^2h_{44,4}^2}{(\lambda_1 - \lambda_2)^2(\lambda_1 - \lambda_3)(\lambda_2 - \lambda_3)} \\
		& + \frac{2h_{12,3}^2}{(\lambda_1 - \lambda_3)(\lambda_2 - \lambda_3)} +
		\frac{2h_{12,4}^2}{(\lambda_1 - \lambda_4)(\lambda_2 - \lambda_4)} - R_{1212},
\end{split}
\end{equation}
\begin{equation}\label{equ2}
\begin{split}
X_{13} &=
		-\frac{(\lambda_2 - \lambda_4)\left[(\lambda_1 - \lambda_2)^2 (\lambda_2 - \lambda_3)+(\lambda_1 - \lambda_4)^2(\lambda_3 - \lambda_4) \right]h_{44,1}^2}{(\lambda_1 - \lambda_2)^2(\lambda_1 - \lambda_3)^2(\lambda_2 - \lambda_3)^2}\\
& + \frac{(\lambda_2 - \lambda_4)\left[(\lambda_1 - \lambda_2)(\lambda_2 - \lambda_3)^2 - (\lambda_1 - \lambda_4)(\lambda_3 - \lambda_4)^2\right]h_{44,3}^2}{(\lambda_1 - \lambda_2)^2(\lambda_1 - \lambda_3)^2(\lambda_2 - \lambda_3)^2}\\
&+ \frac{(\lambda_1 - \lambda_4)(\lambda_2 - \lambda_4)^2(\lambda_3 - \lambda_4)h_{44,2}^2}{(\lambda_1 - \lambda_2)^2(\lambda_1 - \lambda_3)^2(\lambda_2 - \lambda_3)^2}
- \frac{(\lambda_2 - \lambda_4)^2h_{44,4}^2}{(\lambda_1 - \lambda_2)(\lambda_1 - \lambda_3)^2(\lambda_2 - \lambda_3)} \\
		&- \frac{2h_{13,2}^2}{(\lambda_1 - \lambda_2)(\lambda_2 - \lambda_3)} + \frac{2h_{13,4}^2}{(\lambda_1 - \lambda_4)(\lambda_3 - \lambda_4)} - R_{1313},
\end{split}
\end{equation}
\begin{equation}\label{equ3}
\begin{split}
X_{14} &=
		\frac{\left[(\lambda_1 - \lambda_3)^2(\lambda_3 - \lambda_4) - (\lambda_1 - \lambda_2)^2(\lambda_2 - \lambda_4)\right]h_{44,1}^2}{(\lambda_1 - \lambda_2)^2(\lambda_1 - \lambda_3)^2(\lambda_2 - \lambda_3)}\\
&- \frac{(\lambda_3 - \lambda_4)h_{44,2}^2}{(\lambda_1 - \lambda_2)^2(\lambda_1 - \lambda_3)}
- \frac{(\lambda_2 - \lambda_4)h_{44,3}^2}{(\lambda_1 - \lambda_2)(\lambda_1 - \lambda_3)^2}\\
&+ \frac{\left[(\lambda_1 - \lambda_2)(\lambda_2 - \lambda_4)^2
			 - (\lambda_1 - \lambda_3)(\lambda_3 - \lambda_4)^2\right]h_{44,4}^2}{(\lambda_1 - \lambda_2)^2(\lambda_1 - \lambda_3)^2(\lambda_2 - \lambda_3)} \\
		& - \frac{2h_{14,2}^2}{(\lambda_1 - \lambda_2)(\lambda_2 - \lambda_4)} - \frac{2h_{14,3}^2}{(\lambda_1 - \lambda_3)(\lambda_3 - \lambda_4)} -R_{1414},
\end{split}
\end{equation}
\begin{equation}\label{equ4}
\begin{split}
X_{23}&=
	 \frac{(\lambda_1 - \lambda_4)^2(\lambda_2 - \lambda_4)(\lambda_3 - \lambda_4)h_{44,1}^2}{(\lambda_1 -  \lambda_2)^2(\lambda_1 - \lambda_3)^2(\lambda_2 - \lambda_3)^2}+ \frac{(\lambda_1 - \lambda_4)^2h_{44,4}^2}{(\lambda_1 - \lambda_2)(\lambda_1 - \lambda_3)(\lambda_2 - \lambda_3)^2}\\
&- \frac{(\lambda_1 - \lambda_4)\left[(\lambda_2 - \lambda_4)^2(\lambda_3 - \lambda_4) + (\lambda_1 - \lambda_2)^2(\lambda_1 - \lambda_3)\right]h_{44,2}^2}{(\lambda_1 - \lambda_2)^2(\lambda_1 - \lambda_3)^2(\lambda_2 - \lambda_3)^2} \\
  	& - \frac{(\lambda_1 - \lambda_4)\left[(\lambda_2 - \lambda_4)(\lambda_3 - \lambda_4)^2 +(\lambda_1 - \lambda_2) (\lambda_1 - \lambda_3)^2\right]h_{44,3}^2}{(\lambda_1 - \lambda_2)^2(\lambda_1 - \lambda_3)^2(\lambda_2 - \lambda_3)^2}\\
	& + \frac{2h_{23,1}^2}{(\lambda_1 - \lambda_2)(\lambda_1 - \lambda_3)} + \frac{2h_{23,4}^2}{(\lambda_2 - \lambda_4)(\lambda_3 - \lambda_4)} -R_{2323},
\end{split}
\end{equation}
\begin{equation}\label{equ5}
\begin{split}
X_{24} &=-\frac{(\lambda_3 - \lambda_4)h_{44,1}^2}{(\lambda_1 - \lambda_2)^2(\lambda_2 - \lambda_3)}
+ \frac{(\lambda_1 - \lambda_4)h_{44,3}^2}{(\lambda_1 - \lambda_2)(\lambda_2 - \lambda_3)^2}\\
&+ \frac{\left[(\lambda_3 - \lambda_4)(\lambda_2 - \lambda_3)^2
- (\lambda_1 - \lambda_4)(\lambda_1 - \lambda_2)^2\right]h_{44,2}^2}{(\lambda_1 - \lambda_2)^2(\lambda_1 - \lambda_3)(\lambda_2 - \lambda_3)^2} \\
&- \frac{\left[(\lambda_2 - \lambda_3)(\lambda_3 - \lambda_4)^2 + (\lambda_1 - \lambda_2)(\lambda_1 -    \lambda_4)^2\right]h_{44,4}^2}{(\lambda_1 - \lambda_2)^2(\lambda_1 - \lambda_3)(\lambda_2 - \lambda_3)^2} \\
  & + \frac{2h_{24,1}^2}{(\lambda_1 - \lambda_2)(\lambda_1 - \lambda_4)} - \frac{2h_{24,3}^2}{(\lambda_2 - \lambda_3)(\lambda_3 - \lambda_4)} -R_{2424},
\end{split}
\end{equation}
\begin{equation}\label{equ6}
\begin{split}
X_{34} &=
\frac{(\lambda_2 - \lambda_4)h_{44,1}^2}{(\lambda_1 - \lambda_3)^2(\lambda_2 - \lambda_3)}
+ \frac{(\lambda_1 - \lambda_4)h_{44,2}^2}{(\lambda_1 - \lambda_3)(\lambda_2 - \lambda_3)^2}\\
&+ \frac{\left[(\lambda_2 - \lambda_3)^2(\lambda_2 - \lambda_4)
  	- (\lambda_1 - \lambda_3)^2(\lambda_1 - \lambda_4)\right]h_{44,3}^2}{(\lambda_1 - \lambda_2)(\lambda_1 - \lambda_3)^2(\lambda_2 - \lambda_3)^2} \\
  	& + \frac{\left[(\lambda_2 - \lambda_3)(\lambda_2 - \lambda_4)^2 - (\lambda_1 - \lambda_3)(\lambda_1 - \lambda_4)^2\right]h_{44,4}^2}{(\lambda_1 - \lambda_2)(\lambda_1 - \lambda_3)^2(\lambda_2 - \lambda_3)^2} \\
  	& + \frac{2h_{34,1}^2}{(\lambda_1 - \lambda_3)(\lambda_1 - \lambda_4)} + \frac{2h_{34,2}^2}{(\lambda_2 - \lambda_3)(\lambda_2 - \lambda_4)} -R_{3434}.
\end{split}
\end{equation}
Now we define the $3$-form $\Phi$ as follows,
$$\Phi=\sum_{i<j}(\lambda_i+\lambda_j)\theta_{ij}.$$
Clearly, the $3$-form $\Phi$ is globally well defined on $M^4$. To calculate $d\Phi$, we need the following equations,
\begin{equation}\label{equ7}
\begin{split}
&\frac{ \lambda_{1}+ \lambda_{2}}{\left(\lambda_{1}-\lambda_{3}\right) \left(\lambda_{2}-\lambda_{3}\right)}-\frac{ \lambda_{1}+\lambda_{3}}{\left(\lambda_{1}-\lambda_{2}\right) \left(\lambda_{2}-\lambda_{3}\right)}+\frac{ \lambda_{2}+ \lambda_{3}}{\left(\lambda_{1}-\lambda_{2}\right) \left(\lambda_{1}-\lambda_{3}\right)}=0,\\
&\frac{\lambda_{1}+\lambda_{2}}{\left(\lambda_{2}-\lambda_{4}\right) \left(\lambda_{1}-\lambda_{4}\right)}-\frac{\lambda_{1}+\lambda_{4}}{\left(\lambda_{1}-\lambda_{2}\right) \left(\lambda_{2}-\lambda_{4}\right)}+\frac{\lambda_{2}+\lambda_{4}}{\left(\lambda_{1}-\lambda_{2}\right) \left(\lambda_{1}-\lambda_{4}\right)}
=0,\\
&\frac{\lambda_{1}+\lambda_{3}}{\left(\lambda_{1}-\lambda_{4}\right) \left(\lambda_{3}-\lambda_{4}\right)}-\frac{\lambda_{1}+\lambda_{4}}{\left(\lambda_{3}-\lambda_{4}\right) \left(\lambda_{1}-\lambda_{3}\right)}+\frac{\lambda_{3}+\lambda_{4}}{\left(\lambda_{1}-\lambda_{3}\right) \left(\lambda_{1}-\lambda_{4}\right)}
=0,\\
&\frac{\lambda_{2}+\lambda_{3}}{\left(\lambda_{2}-\lambda_{4}\right) \left(\lambda_{3}-\lambda_{4}\right)}-\frac{\lambda_{2}+\lambda_{4}}{\left(\lambda_{2}-\lambda_{3}\right) \left(\lambda_{3}-\lambda_{4}\right)}+\frac{\lambda_{3}+\lambda_{4}}{\left(\lambda_{2}-\lambda_{3}\right) \left(\lambda_{2}-\lambda_{4}\right)}
=0,
\end{split}
\end{equation}
and
\begin{equation}\label{equ8}
\sum_{i<j}(\lambda_i+\lambda_j) R_{ijij}=-f_3,~~\lambda_{i}^{3}+\sigma_4(\lambda_{i})-\frac{1}{2} \lambda_{i} f_{2} =\frac{1}{3}{f_3},
\end{equation}
where the meaning of $\sigma_4(\lambda_{i})=\lambda_1\cdots \hat{\lambda_i}\cdots \lambda_4$ is to remove $\lambda_i$, for example, $$\sigma_4(\lambda_{1})=\lambda_2\lambda_3\lambda_4,~~
\sigma_4(\lambda_{3})=\lambda_1\lambda_2\lambda_4.$$
By a relatively long calculation process,  combining (\ref{equ1}),$\cdots,$ (\ref{equ7}) and (\ref{equ8}), we can obtain the following equation,
\begin{equation}\label{integ}
d\Phi=f_3\Big(\sum_ic_ih_{44,i}^2+1\Big)\omega_1\wedge\omega_2\wedge\omega_3\wedge\omega_4=f_3\Big(\sum_ic_ih_{44,i}^2+1\Big)dM,
\end{equation}
where $$c_i=\frac{2(3f_2-4\lambda_i^2)}{3(\lambda_1-\lambda_2)^2(\lambda_1-\lambda_3)^2(\lambda_2-\lambda_3)^2}.$$

\begin{PROPOSITION}\label{pro2}
Let $x: M^4\to \mathbb{S}^5$ be a closed minimal  hypersurface  with constant squared length of
the second fundamental form $S$  in the $5$-dimensional sphere.
If the $3$-mean curvature $H_3$ is constant  and the number of distinct principal curvatures  $g=4$, then $H_3=0$.
\end{PROPOSITION}
\begin{proof}
Since $g=4$, we can define globally the $3$-form $\Phi$.

Since $$3 f_2=3 \lambda_i^2+3 \sum_{j \neq i}\lambda_j^2 \geq 3 \lambda_i^2+(\sum_{j \neq i}\lambda_j)^2=4 \lambda_{i}^{2},$$
then
$$c_i=\frac{2(3f_2-4\lambda_i^2)}{3(\lambda_1-\lambda_2)^2(\lambda_1-\lambda_3)^2(\lambda_2-\lambda_3)^2}\geq 0,~~i=1,2,3,4.$$
Since $H_3$ is constant, then $f_3$ is constant. By (\ref{integ}), we get
$$0=\int_{M^n}f_3\Big(\sum_ic_ih_{44,i}^2+1\Big)dM=f_3\int_{M^n}\Big(\sum_ic_ih_{44,i}^2+1\Big)dM.$$
Since
$$\Big(\sum_ic_ih_{44,i}^2+1\Big)>0,$$
Then $f_3=0$. By (\ref{niu}), we get $H_3=0$.
\end{proof}
Now we prove Theorem \ref{th1}. Since $g$ is constant, if $g\leq 3$, then $M^4$ is an isoparametric hypersurface by Proposition \ref{pro1}.
If $g=4$, then $H_3=0$ by Proposition \ref{pro2}, and $M^4$ is isoparametric hypersurface by Theorem \ref{wei} in \cite{dgw}. Thus we complete the proof of Main Theorem \ref{th1}.

{\bf Acknowledgements:}  Authors are supported by the
grant No. 12071028  of NSFC.


\begin{thebibliography}{11}
\bibitem{ad}S. C. de Almeida, F. G. B. Brito, {\it Closed 3-dimensional hypersurfaces with constant
mean curvature and constant scalar curvature,} Duke Math. J., 61 (1990), 195-206.
\bibitem{de} S. C. de Almeida, F. G. B. Brito, M. Scherfner, S. Weiss, {\it On CMC hypersurfaces in $S^{n+1}$ with constant Gau$\beta$-Kronecker curvature,} Adv. Geom., 18 (2018), 187-192.
\bibitem{chang2} S. P. Chang, {\it On minimal hypersurfaces with constant scalar curvatures in $S^4$,} J. Differential Geom., 37 (1993), 523-534.
\bibitem{cheng} Q. M. Cheng, S. Ishikawa, {\it A characterization of the Clifford torus,} Proc. Amer. Math. Soc., 127 (1999), 819-828.
\bibitem{chern} S. S. Chern, M. Do Carmo, S. Kobayashi, {\it Minimal submanifolds of a sphere with fundamental form of constant length,} in: F. E. Browder(Ed.), Functional Analysis and Related Fields, Springer, New York, 1970, 59-75.
\bibitem{ding} Q. Ding, Y. L. Xin, {\it On Chern's problem for rigidity of minimal hypersurfaces in the spheres,} Adv. Math., 227 (2011),  131-145.
\bibitem{dgw} Q. T. Deng, H. Gu, Q. Y. Wei, {\it Closed Willmore minimal hypersurfaces with constant scalar curvature in $S^5(1)$ are isoparametric,} Adv. Math., 314 (2017),  278-305.
\bibitem{ge} J. Q. Ge, Z. Z. Tang, {\it Chern conjecture and isoparametric hypersurfaces,} Differential Geometry, Adv. Lect. Math., vol. 22, International Press, Somerville, MA, 2012, 49-60.
\bibitem{li3}H. Z. Li, {\it Global rigidity theorems of hypersurfaces,} Ark. Mat., 35 (1997), 327-351.
\bibitem{lif}F. G. Li, {\it A note on the Chern Conjecture in dimension four,} Differential Geom. Appl., 84 (2022), 101928.
\bibitem{LS} T. Lusala, M. Scherfner, L. A. M. Sousa, Jr., {\it Closed minimal Willmore hypersurfaces of $S^5(1)$ with constant scalar curvatures,} Asian J. Math., 9 (1) (2005), 65-78.
\bibitem{L} H. B. Lawson, {\it Local rigidity theorems for minimal hypersurfaces,} Ann. of Math. 89 (2) (1969), 187-191.
\bibitem{peng} C. K. Peng, C. L. Terng, {\it Minimal hypersurfaces of spheres with constant scalar curvature,} Ann. of Math. Stud., 103 (1983), 177-198.
\bibitem{peng1} C. K. Peng, C. L. Terng, {\it The scalar curvature of minimal hypersurfaces in spheres,} Math. Ann., 266 (1) (1983), 105-113.
\bibitem{tali}B. Tang, L. Yang, {\it An intrinsic rigidity theorem for closed minimal hypersurfaces in $\mathbb{S}^5$ with constant nonnegative scalar curvature,} Chin. Ann. Math., Ser. B, 39 (5) (2018), 879-888.
\bibitem{tang}Z. Z. Tang, D. Y. Wei, W. J. Yan, {\it A sufficient condition for a hypersurface to be isoparametric,} Tohoku Math. J., 72 (2000), 493-505.
\bibitem{tang1}Z. Z. Tang, W. J. Yan, {\it On the Chern conjecture for isoparametric hypersurfaces,} arXiv:2001.10134.
\bibitem{weixu} S. M. Wei, H. W. Xu, {\it Scalar curvature of minimal hypersurfaces in a sphere,} Math. Res. Lett., 14 (3) (2007) 423-432.
\bibitem{scher} M. Scherfner, S. Weiss, S. T. Yau, {\it A review of the Chern conjecture for isoparametric hypersurfaces in spheres,} in: Advances in Geometric Analysis, Adv. Lect. Math. (ALM), vol. 21, Int. Press, Somerville, MA, 2012, pp. 175-187.
\bibitem{sim} J. Simons, {\it Minimal varieties in Riemannian manifolds,}  Ann. of Math., 88 (1968), 62-105.
\bibitem{yang} Y. J. Suh, H. Y. Yang, {\it The scalar curvature of minimal hypersurfaces in a unit sphere,}  Commun. Contemp. Math., 9 (2) (2007), 183-200.
\bibitem{yang2}H. Yang, Q. M. Cheng, {\it Chern's conjecture on minimal hypersurfaces,} Math. Z., 227 (1998), 377-390.
\bibitem{yang3} H. C. Yang, Q. M. Cheng, {\it An estimate of the pinching constant of minimal hypersurfaces with constant scalar curvature in the unit sphere,} Manuscripta Math., 84 (1994), 89-100.
\bibitem{zhang} Q. Zhang, {\it The pinching constant of minimal hypersurfaces in the unit spheres,} Proc. Amer. Math. Soc., 138 (2010), 1833-1841.
\end{thebibliography}
\end{document}